\documentclass{article}
\usepackage{amssymb}
\usepackage{color}
\newcommand{\ord}{\mathrm{ord}\,}
\newcommand{\cont}{\mathrm{cont}}

\newcommand{\Zer}{\mathrm{Zer}}
\newcommand{\bC}{\mathbf{C}}
\newcommand{\bN}{\mathbf{N}}
\newcommand{\bR}{\mathbf{R}}
\newcommand{\bQ}{\mathbf{Q}}
\newcommand{\supp}{\mathrm{supp}\,}
\newcommand{\jac}{\mathrm{jac}\,}
\newcommand{\Nj}{{\cal N}_J}

\newcommand{\Teis}[2]{
   \setlength{\unitlength}{1ex}
   \begin{picture}(2,0)(0,0.4)
      \put(0,1.1){\line(1,0){2}}
      \put(0,0.9){\line(1,0){2}}
      \put(1,1.2){\makebox(0,0)[b]{$\scriptstyle #1$}}
      \put(1,0.8){\makebox(0,0)[t]{$\scriptstyle #2$}}
   \end{picture}}

\newcommand{\Teisssr}[4]{
   \setlength{\unitlength}{1ex}
   \begin{picture}(#3,3)(0,0.4)
      \put(0,1.15){\line(1,0){#3}}
      \put(0,0.85){\line(1,0){#3}}
      \put(#4,1.3){\makebox(0,0)[b]{$#1$}}
      \put(#4,0.7){\makebox(0,0)[t]{$#2$}}
   \end{picture}}

\newtheorem{Theorem}{Theorem}
\newtheorem{Lemma}{Lemma}
\newtheorem{Corollary}{Corollary}
\newtheorem{Remark}{Remark}

\newtheorem{Property}{Property}

\newtheorem{Proposition}{Proposition}
\newtheorem{Definition}{Definition}
\newtheorem{Example}{Example}

\newenvironment{proof}[1][Proof]{\textbf{#1.} }{\
\rule{0.5em}{0.5em}}

\title{On the approximate jacobian Newton diagrams of an irreducible
plane curve  \footnotetext{
     \noindent   \begin{minipage}[t]{4.5in}
       {\small
       2000 {\it Mathematics Subject Classification:\/} Primary 32S55;
       Secondary 14H20.\\
       Key words and phrases: irreducible plane curve, approximate root,
       jacobian Newton diagram, polar quotient.\\
       The first-named author was partially supported by the Spanish Project
       PNMTM 2007-64007.}
       \end{minipage}}}

\author{Evelia R.\ Garc\'{\i}a Barroso and Janusz Gwo\'zdziewicz}

\begin{document}
\maketitle

\begin{abstract}
\noindent We introduce the notion of an approximate jacobian Newton
diagram which is the jacobian Newton diagram of the morphism
$(f^{(k)},f)$, where $f$ is a branch and $f^{(k)}$ is a
characteristic approximate root of $f$. We prove that the set of all
approximate jacobian Newton diagrams is a complete topological
invariant. This generalizes  theorems of Merle and Ephraim about the
decomposition of the polar curve of a branch.
\end{abstract}

\section{Introduction}

\noindent Every two complex series $f,g\in\bC\{x,y\}$ such that
$f(0,0)=g(0,0)=0$ define 
a germ of a holomorphic mapping
$(g,f):(\bC^2,0)\longrightarrow (\bC^2,0)$.
Assume that the curves $f=0$ and $g=0$ share no common component. 
Then the critical locus of this mapping is a  germ of an analytic curve
and its  direct image by $(g,f)$ is also an analytic curve called the 
\textit{discriminant curve}. Let $D(u,v)=0$ be an equation of 
the discriminant curve in the coordinates $(u,v)=(g(x,y),f(x,y))$. 
We call the Newton diagram of $D(u,v)$ the \textit{jacobian Newton
diagram} of the morphism $(g,f)$ and denote it~$\Nj(g,f)$.  

\noindent Note that if $g=0$ is a smooth curve transverse to $f=0$ 
then $\Nj(g,f)$ is  the jacobian Newton diagram 
of the curve $f=0$ introduced in \cite{Teissier3}. 
With these assumptions Teissier proves in~\cite{Teissier1} 
that $\Nj(g,f)$ depends only on the topological type of the curve $f=0$. 

\noindent 
Merle in \cite{Merle} studies the case of a smooth curve $g=0$ 
transverse to an irreducible singular curve $f=0$.
He gives a description of the jacobian Newton diagram in 
terms of other invariants of singularity of a curve $f=0$. 
He also shows that the datum of the jacobian Newton diagram determines 
the equisingularity class of the curve (or equivalently its embedded topological type).  Ephraim in \cite{Ephraim} extends Merle's result to any smooth curve $g=0$. 

\medskip
\noindent 
Let $f$ be an irreducible Weierstrass polynomial. 
In this paper we generalize the results of Merle to
the family $\{{\cal N}_J(f^{(k)},f)\}_k$, where $f^{(k)}$ is the
$k$-th characteristic approximate root of $f$ introduced in
\cite{Abhyankar-Moh}. We prove, in two different ways, that this family is a complete topological invariant
of the branch $f=0$. Our computations are based on the decomposition of the critical locus of 
the mapping $(f^{(k)},f)$, which is analogous to the decomposition of the polar curve obtained by Merle in
\cite{Merle}.  

\section{Plane branches, semigroup and approximate roots}
\label{plane-branches}

\medskip
\noindent
We mean by the fractional power series the elements of the ring $\bC\{x\}^*=\bigcup_{n\in \bN}\bC\{x^{1/n}\}$. 
For every two fractional power series  $\delta$ and~$\delta'$ 
we call the number ${\cal O}(\delta,\delta')=\ord_x(\delta(x)-\delta'(x))$ 
the {\em contact order} between $\delta$ and $\delta'$. 

\medskip
\noindent
Every convergent power series $g(x,y)\in\bC\{x,y\}$, $g(0,0)=0$ has a Newton-Puiseux factorization
$$ g(x,y)=u(x,y)x^N\prod_{i=1}^d (y-\gamma_i(x)), 
$$
where $u(x,y)\in\bC\{x,y\}$, $u(0,0)\neq0$,  $N$ is a nonnegative integer
 and $\gamma_i(x)$ are fractional power series  of positive order. 
We will call $\gamma_i$ the Newton-Puiseux roots of~$g$ and denote the set 
$\{\gamma_1,\dots,\gamma_d\}$  by $\Zer g$.

\medskip
\noindent 
Let $f(x,y)$ be an irreducible power series such that $\ord_y(f(0,y))=n\geq 1$.
Then $f$ has a Newton-Puiseux root of the form 
$\gamma_1(x)=\sum_{i=1}^{\infty} a_ix^{i/n}$. 
The other Newton-Puiseux roots are 
$\gamma_j(x)=\sum_{i=1}^{\infty} a_i\omega^{(j-1)i}x^{i/n}$ for
$1\leq j\leq n$, where $\omega\in\bC$ is an $n$-th primitive root of unity.
 The contact orders between the elements of $\Zer f$ form a set 
$\{ {b_1}/{n},\dots, {b_g}/{n}\}$, where $b_1<b_2<\dots<b_g$ and 
$\gcd(n,b_1,\dots,b_g)=1$. 
We put $b_0=n$ and call the sequence $(b_0,b_1,\dots,b_g)$
the {\em Puiseux characteristic} of $f$. By convention $b_{g+1}=+\infty$.

\medskip

\noindent Let  $A$ and $B$ be finite sets of fractional power series.
The {\em contact} $\cont(A,B)$ is by definition
$\max\{{\cal O}(\alpha,\beta):\; \alpha\in A, \, \beta \in B\}$.
If $\alpha(x)$ is a fractional power series and $f(x,y)$, $g(x,y)$ are irreducible 
power series co-prime to $x$ then by abuse of notation we will write 
$\cont(\alpha,f):=\cont(\{\alpha\},\Zer f)$  and  
$\cont(f,g):=\cont(\Zer f ,\Zer g)$.

\medskip
\noindent It is well-known (see for example Lemma 4.3 of \cite{Casas})
that for every Newton-Puiseux root $\alpha$ of $f$  we have
$\cont(\alpha,g)=\cont(f,g)$. The contact between irreducible 
power series has a strong triangle inequality property:
if $h_i\in\bC\{x,y\}$  for $ i=1,2,3$ are irreducible power series co-prime to $x$ then 
$\cont(h_1,h_2)\geq \min(\cont(h_1,h_3),$ $\cont(h_2,h_3))$.

\medskip
\noindent In \cite{Abhyankar-Moh} the authors introduce the concept
of {\em approximate root} as a consequence of the following proposition:

\begin{Proposition}
Let $\mathbf A$ be an integral domain. If $f(y)\in \mathbf A[y]$ is
monic of degree $d$ and $p$ is invertible in $\mathbf A$ and divides
$d$, then there exists a unique monic polynomial $g(y) \in \mathbf A[y]$
such that the degree of $f-g^p$ is less than $d-\frac{d}{p}$.
\end{Proposition}

\medskip

\noindent This allows us to define:

\begin{Definition}
The unique monic polynomial of the preceding proposition is called
the $p$-th approximate root of $f$.
\end{Definition}

\noindent Let $f\in\bC\{x\}[y]$ be an irreducible Weierstrass polynomial 
with Puiseux characteristic $(b_0,\ldots,b_g)$. 
Put $l_k:=\gcd(b_0,\ldots,b_k)$. In particular $l_k$ divides 
$\deg f=b_0$ for all $k\in \{0,\ldots,g\}$.
In the sequel for $k\in \{0,\ldots,g-1\}$ we denote $f^{(k)}$ the $l_k$-th approximate root of $f$ 
and we call these polynomials the {\em characteristic approximate} roots of $f$. 
By convention we put $f^{(-1)}=x$.

\medskip

\noindent The following proposition is the main one in
\cite{Abhyankar-Moh} (see also \cite{GP2} and \cite{Popescu}):

\begin{Proposition}
\label{ch-prop} 
Let $f\in\bC\{x\}[y]$ be an irreducible Weierstrass polynomial 
with Puiseux characteristic $(b_0,\ldots,b_g)$. 
Then the characteristic approximate roots~$f^{(k)}$  for $k\in \{0,\ldots,g-1\}$,
have the following properties:

\begin{enumerate}

\item The polynomial $f^{(k)}$ is irreducible
with Puiseux characteristic 
$(b_0/l_k,\ldots,b_k/l_k)$.

\item The $y$-degree of $f^{(k)}$ is equal to $b_0/l_k$ and
$\cont(f,f^{(k)})=b_{k+1}/b_0$.

\end{enumerate}
\end{Proposition}

\begin{Example}\label{example1} 
Take the irreducible Weierstrass polynomial
$f=(y^3-6x^3y-x^4)^2-9x^9$ of Puiseux characteristic $(6,8,11)$.
The characteristic approximate roots of~$f$ are 
$f^{(0)}=y$ and $f^{(1)}=y^3-6x^3y-x^4$. 
The Newton-Puiseux roots of $f$ are of the form 
$y=\omega^8x^{4/3}+2\omega^{10}x^{5/3}+\omega^{11}x^{11/6}+\cdots$, 
where $\omega^6=1$ while 
the Newton-Puiseux roots of $f^{(1)}$ are 
$y=\epsilon^4x^{4/3}+2\epsilon^5x^{5/3}-\frac{8}{3}x^2+\cdots$, where $\epsilon^3=1$.  
One can check directly 
that $\cont(f,f^{(0)})=8/6$ and $\cont(f,f^{(1)})=11/6$.
\end{Example}

\section{Jacobian Newton diagrams}
\label{jacobian-newton-diagram}
\medskip

\noindent In this section we recall the notion of the jacobian
Newton diagrams and we establish some preliminary results which are
necessary for the next.

\medskip

\noindent Write $\bR_{+}=\{\,x\in\bR: x\geq0\,\}$. Let
$f\in\bC\{x,y\}$, $f(x,y)=\sum a_{i,j}x^iy^j$ be a non-zero
convergent power series. Put $\supp f:=\{\,(i,j): a_{i,j} \neq0\,\}$
the {\em support} of $f$. By definition the {\em Newton diagram} of
$f$ in the coordinates $(x,y)$  is
$$
\Delta_f:= \mbox{Convex Hull}\;(\supp f+\bR_{+}^2).
$$

\noindent An important property of Newton diagrams is that the
Newton diagram of a product is the Minkowski sum of Newton diagrams.
One has $\Delta_{fg}=\Delta_f+\Delta_g$, where
$\Delta_f+\Delta_g=\{\,a+b:a\in\Delta_f, b\in\Delta_g\,\}$. In
particular if $f$ and $g$ differ by an invertible factor $u\in
\bC\{x,y\}$, $u(0,0)\neq0$ then $\Delta_f=\Delta_g$. Thus the
Newton diagram of a plane analytic curve is well defined because an
equation of an analytic curve is given up to invertible factor,  where an analytic plane curve is a principal ideal  of the ring of convergent power series 
$\mathbf C\{x,y\}$, which we will denote  by $f(x,y)=0$.
We will write $\Delta_{f=0}$ for the Newton diagram of the curve~$f=0$.

\medskip

\noindent Following Teissier \cite{Teissier2} we introduce {\em
elementary Newton diagrams}. For $m,n>0$  we put
$\{\Teis{n}{m}\}=\Delta_{x^n+y^m}$. We put also
$\{\Teis{n}{\infty}\}=\Delta_{x^n}$ and
$\{\Teis{\infty}{m}\}=\Delta_{y^m}$.

\medskip

\noindent Every Newton diagram $\Delta \subsetneq \bR_{+}^2$ has a
unique representation $\Delta=\sum_{i=1}^r\left\{
\Teis{L_i}{M_i}\right\}$, where {\em inclinations} of successive
elementary diagrams form an increasing sequence (by definition the
inclination of  $\left\{\Teis{L}{M}\right\}$ is $L/M$ with the
conventions that $L/\infty=0$ and $\infty/M=+\infty$). We shall call
this representation the {\em canonical decomposition} of $\Delta$. 

\medskip

\noindent Let $\sigma=(g,f):(\bC^2,0)\to(\bC^2,0)$ be an analytic
mapping given by $\sigma(x,y)=(g(x,y),f(x,y)):=(u,v)$ and such that
$\sigma^{-1}(0,0)=\{(0,0)\}$.
Then every local analytic curve $h(x,y)=0$ has a well-defined 
{\em direct image} $\sigma^{*}(h=0)$
which is an analytic curve in the target space (see \cite{Casas-Asian}).
The Newton diagram of the direct image is characterized by two properties: 
\begin{enumerate}

\item  If $h$ is an irreducible power series then 
         $\Delta_{\sigma^{*}(h=0)}=\left\{\Teisssr{(f,h)_0}{(g,h)_0}{7}{3.5}\right\}$,
        where~$(r,s)_0$ denotes the intersection multiplicity of the curves $r=0$ and $s=0$ at the origin.

\item If $h=h_1h_2$ then 
         $\Delta_{\sigma^{*}(h=0)}=\Delta_{\sigma^{*}(h_1=0)}+\Delta_{\sigma^{*}(h_2=0)}$.

\end{enumerate}

\noindent
Let $\jac(g,f)=\frac{\partial
g}{\partial x}\frac{\partial f}{\partial y} - \frac{\partial
g}{\partial y}\frac{\partial f}{\partial x}$ be the jacobian 
determinant of the mapping $\sigma$. The direct image
(see Preliminaries in \cite{Casas-Asian}) of 
$\jac(g,f)=0$ by $\sigma$ is called the {\em discriminant curve}. 
We will write $\Nj(g,f)$ for the  Newton diagram of the discriminant curve  
and following Teissier (see \cite{Teissier3}) call it  the {\em jacobian Newton diagram} 
of the morphism $\sigma=(g,f)$.

\section{Approximate jacobian Newton diagrams of a
branch} \label{computation}

\medskip

\noindent In this section we introduce the notion of the approximate
jacobian Newton diagrams of an irreducible plane curve and we
compute them. In what follows a branch $f(x,y)=0$ will be given by an
irreducible Weierstrass polynomial.

\medskip
\noindent Let $f$ be an irreducible Weierstrass polynomial and let
$f^{(k)}$, for $0\leq k \leq g-1$, be the characteristic
approximate roots of $f$. The jacobian Newton diagram  ${\cal
N}_J(f^{(k)},f)$ is called the $k$-th {\em approximate jacobian Newton
diagram of the branch} $f(x,y)=0$.

\medskip
\noindent The following result about the factorization of the jacobian
$\jac(f^{(k)},f)$ is the main result of this note:

\begin{Theorem}\label{decomposition} 
Let $f\in\bC\{x\}[y]$ be an irreducible Weierstrass polynomial 
with Puiseux characteristic $(b_0,\ldots,b_g)$.
Let $f^{(k)}$, $0\leq k\leq g-1$, be the $k$-th characteristic approximate root of $f$. 
Then  the jacobian $\jac(f^{(k)},f)$ admits a factorization
$$\jac(f^{(k)},f)=\Gamma^{(k+1)}\cdots \Gamma^{(g)},
$$
where the factors $\Gamma^{(i)}$ are not necessary irreducible, 
$x$ is co-prime to the product $\Gamma^{(k+2)}\cdots \Gamma^{(g)}$  and such that
\begin{enumerate}
\item If $\alpha$ is a Newton-Puiseux root of $\Gamma^{(k+1)}$  then  $\cont(\alpha,f)<b_{k+1}/b_0$.
\item If $\alpha$ is a Newton-Puiseux root of $\Gamma^{(i)}$, $k+2 \leq i\leq g$ then $\cont(\alpha,f)=b_i/b_0$.
\item The intersection multiplicity 
$(\Gamma^{(i)},x)_0=n_1\cdots n_{i-1}(n_i-1)$ for $k+2 \leq i\leq g$.
\end{enumerate}
\end{Theorem}

\medskip

\noindent The proof of Theorem~\ref{decomposition} will be done in
Section~\ref{proof-main-result}.

\medskip

\noindent
The contacts between Newton-Puiseux roots of $\Gamma^{(k+1)}$ and $f$ 
are not determined by the Puiseux characteristic of $f$ as the following example shows.

\begin{Example}
Let $f=(y^3-6x^3y-x^4)^2-9x^9$ be the Weierstrass polynomial from Example~\ref{example1}
and let $g=(y^3-x^4)^2+x^9-x^7y^2$. 
Both series $f$ and $g$ are irreducible with the same Puiseux characteristic $(6,8,11)$. 
The jacobian $\jac(f^{(1)},f)=243x^8(y^2-2x^3)$ has two Newton-Puiseux roots 
$\alpha_1(x)=\sqrt{2}x^{3/2}+\cdots$, $\alpha_2(x)=-\sqrt{2}x^{3/2}+\cdots$ and 
$\cont(\alpha_i,f)=\frac{4}{3}<\frac{b_2}{b_0}$ for $i=1,2$.

\medskip 

\noindent
On the other hand there are four Newton-Puiseux roots 
$\beta_1(x)=0$, 
$\beta_2(x)=\frac{8}{27}x^2+\cdots$, 
$\beta_3(x)=\sqrt{\frac{21}{27}}x+\cdots$
$\beta_4(x)=-\sqrt{\frac{21}{27}}x+\cdots$
of $\jac(g^{(1)},g)=x^6y(21y^3-27x^2y+8x^4)$ and 
$\cont(\beta_i,g)=\frac{4}{3}$ for $i=1,2$, but $\cont(\beta_i,g)=1$ for $i=3,4$.
\end{Example}

\medskip

\noindent Further we will use the following property
of the intersection multiplicity which 
is a consequence of the Noether's formula (see \cite{GP2} Proposition~3.3):

\begin{Property}
\label{Noether}
 Let $g(x,y), \;h(x,y)$ be irreducible power series co-prime to $x$. 
Then for fixed $g$, the function $h\mapsto\frac{(g,h)_0}{(x,h)_0}$ depends only
on the contact $\cont(g,h)$ and is a strictly 
increasing function of this quantity.
\end{Property}

\begin{Corollary}\label{canonical-form} 
Under assumptions and notations of Theorem~\ref{decomposition}
the jacobian Newton diagram of the mapping $(f^{(k)},f)$ has 
the canonical decomposition 
$$ 
\Nj(f^{(k)},f)= \sum_{i=k+1}^g
\left\{\Teisssr{(f,\Gamma^{(i)})_0}
{(f^{(k)},\Gamma^{(i)})_0} {12}{6}\right\}.
$$
\end{Corollary}

\noindent
\begin{proof}
\noindent
We prove that for every irreducible factor $h$ of $\jac(f^{(k)},f)$ the quotient 
$\frac{(f,h)_0}{(f^{(k)},h)_0}$ depends only on the contact $\cont(f,h)$. 
Indeed there are two cases:
if $\cont(f,h)<b_{k+1}/b_0$ then by the strong triangle inequality $\cont(f^{(k)},h)=\cont(f,h)$ hence 
$\frac{(h,f^{(k)})_0}{(x,f^{(k)})_0}=\frac{(h,f)_0}{(x,f)_0}$ and we get 
\begin{equation}\label{Eq:incl1}
\frac{(f,h)_0}{(f^{(k)},h)_0}=\frac{(x,f)_0}{(x,f^{(k)})_0},
\end{equation}

\noindent if $\cont(f,h)>b_{k+1}/b_0$ then also by the strong triangle inequality $\cont(f^{(k)},h)=\cont(f^{(k)},f)$ hence
$\frac{(f^{(k)},h)_0}{(x,h)_0}=\frac{(f^{(k)},f)_0}{(x,f)_0}$ and we get 
\begin{equation}\label{Eq:incl2}
\frac{(f,h)_0}{(f^{(k)},h)_0}=\frac{(x,f)_0}{(f^{(k)},f)_0}\cdot\frac{(f,h)_0}{(x,h)_0} . 
\end{equation}

\medskip

\noindent Fix $i\in\{k+1,\dots,g\}$ and write $\Gamma^{(i)}$ as a product $h_1\cdots h_r$
of irreducible factors $h_j$ for $1\leq j\leq r$. Then the Newton diagram of the direct 
image of the curve $\Gamma^{(i)}=0$ is the sum 
$\sum_{j=1}^r\left\{\Teisssr{(f,h_j)_0}{(f^{(k)},h_j)_0} {10}{5}\right\}$. 
Since all elementary Newton diagrams in the above sum have the same inclination one has 
$$
\sum_{j=1}^r\left\{\Teisssr{(f,h_j)_0}{(f^{(k)},h_j)_0} {10}{5}\right\} = 
\left\{\Teisssr{\sum_{j=1}^r(f,h_j)_0}{\sum_{j=1}^r(f^{(k)},h_j)_0} {18}{9}\right\} = 
\left\{\Teisssr{(f,\Gamma^{(i)})_0}{(f^{(k)},\Gamma^{(i)})_0} {12}{6}\right\} .
$$

\noindent We proved that the jacobian Newton diagram $\Nj(f^{(k)},f)$ is the sum of elementary Newton diagrams 
from the statement of the Corollary. The inclination of the first elementary Newton diagram is given by formula 
(\ref{Eq:incl1}) which can be be written as 
$\frac{(x,f)_0}{(f^{(k)},f)_0}\cdot\frac{(f,f^{(k)})_0}{(x,f^{(k)})_0}$ . The inclinations of the remaining elementary Newton
diagrams are given by formula (\ref{Eq:incl2}). By Property \ref{Noether} these inclinations form a strictly increasing sequence.
This finishes the proof.
\end{proof}

\medskip

\noindent
Now our aim is to give an arithmetical formula for $\Nj(f^{(k)},f)$. 

\medskip

\noindent Put $\overline{b_k}:=(f,f^{(k-1)})_0$ for $k\in\{0,1,\ldots,g\}$.
Following Zariski (see \cite{Zariski}), the set $\{\overline{b_0},
\overline{b_1},\ldots,\overline{b_g}\}$ is a mi\-ni\-mal system of
generators of the {\em semigroup} $$\Gamma(f):=\{(f,g)_0\;:\;f
\;\hbox{\rm  is not a factor of} \;g \}$$

\noindent of the branch $f(x,y)=0$. This system of generators is
uniquely determined by the Puiseux characteristic of $f$ in
the following way: $\overline{b_0}=b_0$, $\overline{b_1}=b_1$ and
$\overline{b_q}=n_{q-1}\overline{b_{q-1}}+b_q-b_{q-1}$ for $2\leq q\leq g$.
Recall that $n_i=l_{i-1}/l_i$, where $l_i=\gcd(b_0,\ldots,b_i)=\gcd(\overline{b_0},\ldots,\overline{b_i})$.

\medskip

\noindent Remember that the {\em Milnor number} of a curve
$g(x,y)=0$ is by definition the intersection multiplicity
$\left(\frac{\partial g}{\partial x}, \frac{\partial g}{\partial
y}\right)_0$.

\medskip
\begin{Theorem}
\label{app-diag} Let $f=0,$ where $f$ is an irreducible Weierstrass polynomial, be a branch  with semigroup $\Gamma(f)=\langle\overline{b_0},\ldots,\overline{b_g}\rangle$. 
Then the canonical decomposition of the $k$-th approximate jacobian
Newton diagram of~$f$ is

\[{\cal
N}_J(f^{(k)},f)= \left\{\Teisssr{l_k\bigl(\mu(f^{(k)})+\overline{m}-1\bigr)}
{\mu(f^{(k)})+\overline{m}-1}{24}{12}\right\}+ \sum_{i=k+2}^g
\left\{\Teisssr{(n_i-1)\overline{b_i}} {\overline{m}n_{k+2}\cdots
n_{i-1}(n_i-1)}{26}{13}\right\},
\]

\noindent where $\overline{m}=\overline{b_{k+1}}/ {l_{k+1}}$,
and $\mu(f^{(k)})$ is the Milnor number of
$f^{(k)}=0$.
\end{Theorem}
\noindent \begin{proof} 
In the course of the proof we shall use the canonical decomposition 
of ${\cal N}_J(f^{(k)},f)$ from Corollary~\ref{canonical-form}. 
We shall express all intersection multiplicities $(f,\Gamma^{(i)})_0$ and $(f^{(k)},\Gamma^{(i)})_0$ 
for $k+1\leq i \leq g$ in terms of the generators of the semigroup $\Gamma(f)$.

\medskip

\noindent First consider $\Gamma^{(i)}$ for $k+2\leq i\leq g$. 
By Theorem~\ref{decomposition} the contact of every irreducible 
factor of $\Gamma^{(i)}$ with $f$ equals $b_i/b_0$. By Property~\ref{Noether} 
and Theorem~\ref{decomposition}:
\begin{equation}
 \label{tenerife}
(f,\Gamma^{(i)})_0=
(x,\Gamma^{(i)})_0\frac{(f,\Gamma^{(i)})_0}{(x,\Gamma^{(i)})_0}=
(x,\Gamma^{(i)})_0\frac{(f,f^{(i-1)})_0}{(x,f^{(i-1)})_0}=
(n_i-1)\overline{b_i} .
\end{equation}

\noindent By Corollary~\ref{canonical-form} and equality~(\ref{Eq:incl2}) 
$$
\frac{(f,\Gamma^{(i)})_0}{(f^{(k)},\Gamma^{(i)})_0}=
\frac{(f,f^{(i-1)})_0}{(f^{(k)},f^{(i-1)})_0}=
\frac{(x,f)_0}{(f^{(k)},f)_0}\cdot\frac{(f,f^{(i-1)})_0}{(x,f^{(i-1)})_0}=
\frac{l_{i-1}\overline{b_i}}{\overline{b_{k+1}}}.
$$
Hence by (\ref{tenerife})
$$
(f^{(k)},\Gamma^{(i)})_0=
\frac{\overline{b_{k+1}}}{l_{i-1}\overline{b_i}}(f,\Gamma^{(i)})_0=
\overline{m}n_{k+2}\cdots n_{i-1}(n_i-1) .
$$

\noindent In order to compute $(f^{(k)},\Gamma^{(k+1)})_0$ we use 
Theorem~3.2 of~\cite{Casas}. We get 
$$ (f^{(k)},\jac(f^{(k)},f))_0=\mu(f^{(k)})+(f^{(k)},f)_0-1 .
$$
\noindent Since $(f^{(k)},\jac(f^{(k)},f))_0 = \sum_{i=k+1}^g (f^{(k)},\Gamma^{(i)})_0$ we have
\begin{eqnarray*}
(f^{(k)},\Gamma^{(k+1)})_0 &=&
 \mu(f^{(k)})+(f^{(k)},f)_0-1-\sum_{i=k+2}^g \overline{m}n_{k+2}\cdots n_{i-1}(n_i-1)  \\
& = & \mu(f^{(k)}) +\overline{b_{k+1}}-1-\overline{m}(l_{k+1}-1)=
\mu(f^{(k)})+\overline{m}-1 .
\end{eqnarray*}

\noindent Finally  by Corollary~\ref{canonical-form} and equality~(\ref{Eq:incl1})
$$
\frac{(f,\Gamma^{(k+1)})_0}{(f^{(k)},\Gamma^{(k+1)})_0}=
\frac{(x,f)_0}{(x,f^{(k)})_0}=l_{k}
$$
Hence $(f,\Gamma^{(k+1)})_0 = l_{k}\bigl( \mu(f^{(k)})+\overline{m}-1 \bigr)$. 
\end{proof}

\begin{Remark}
In the above proof we compute the inclinations of elementary Newton diagrams of the canonical
decomposition of ${\cal N}_J(f^{(k)},f)$ which are equal to  $\frac{l_{i-1}\overline{b_i}}{\overline{b_{k+1}}}$
for $i\in\{k+1,\ldots,g\}$. These inclinations are called {\em jacobian invariants}.
\end{Remark}

\begin{Example}
\label{example2} Let $f(x,y)=(y^2-x^3)^2-x^5y$.
Then 
$f=0$ is a branch and 
$\Gamma(f)=\langle4,6,13\rangle$. 
The characteristic approximate roots of $f$ are $f^{(0)}=y$ and
$f^{(1)}=y^2-x^3$. The factorization of $\jac(f^{(0)},f)$ described 
in Theorem~\ref{decomposition} is 
$\jac(f^{(0)},f)=\Gamma^{(1)} \Gamma^{(2)}$,
where $\Gamma^{(1)}=x^2$ and $\Gamma^{(2)}=6y^2+5x^2y-6x^3$.
We also have $\jac(f^{(1)},f)=x^4(10y^2+3x^3)$. Finally
$\Nj(f^{(0)},f)=\{\Teis{8}{2}\}+\{\Teis{13}{3}\}$ and
$\Nj(f^{(1)},f)=\{\Teis{28}{14}\}$.
\end{Example}
\medskip
\begin{Corollary}
The family of the approximate jacobian
Newton diagrams of a branch only depends on its topological type.
\end{Corollary}

\medskip

\noindent If $f$ is an irreducible Weierstrass polynomial then $f^{(0)}=0$ is a smooth curve. 
By Smith-Merle-Ephraim (see for example Theorem 2.2 of \cite{GB-G2}) the approximate jacobian Newton diagram ${\cal
N}_J(f^{(0)},f)$ determines the topological type of the branch $f=0$. Nevertheless
we can also obtain the generators of the semigroup of the branch $f=0$ using the whole family of its approximate jacobian 
Newton diagrams in an easy way: let
$\Gamma(f)=\langle \overline{b_0},\ldots,\overline{b_g}\rangle$ be the semigroup of $f=0$. It is clear 
that $\overline{b_0}$ is the smallest inclination of ${\cal N}_J(f^{(0)},f)$. 
Denote by ${\iota}$ the inclination of the elementary diagram ${\cal
N}_J(f^{(g-1)},f)$. Put ${\cal H}_r$, for $r\in \{0,\ldots,g-2\}$, the
height of the last elementary diagram of ${\cal
N}_J(f^{(r)},f)$, that is
the height of the elementary diagram of  ${\cal
N}_J(f^{(r)},f)$ which has
the biggest inclination. Then $\overline{b}_{r+1}=\frac{\iota {\cal
H}_r}{\iota-1}$ for $r\in \{0,\ldots,g-2\}$. Finally
$\overline{b}_{g}=\frac{{\cal L}}{\iota -1}$, where ${\cal L}$ is the
length of the last elementary diagram of ${\cal
N}_J(f^{(g-2)},f)$.

\medskip

\begin{Example}
Consider the branches $f_i=0$ for $i\in \{1,\ldots,4\}$ with semigroups
$\Gamma(f_1)=\langle 4,14,31\rangle$, $\Gamma(f_2)=\langle4,6,35\rangle$,
$\Gamma(f_3)=\langle4,6,37\rangle$ and $\Gamma(f_4)=\langle6,10,31\rangle$.
By Theorem \ref{app-diag}  we have 
${\cal
N}_J(f_1^{(1)},f_1)={\cal
N}_J(f_2^{(1)},f_2)=\{\Teis{72}{36}\}$ and
${\cal
N}_J(f_3^{(1)},f_3)={\cal
N}_J(f_4^{(1)},f_4)=\{\Teis{76}{38}\}$.
\end{Example}

\noindent 
Given a branch $f=0$, put ${\cal F}$ its family of approximate jacobian Newton diagrams but 
the first one. The example shows that ${\cal F}$ is not a complete topological invariant of a branch. 
The curves $f_3=0$ and $f_4=0$ have the same ${\cal F}$ but they have different multiplicities at the origin. 
The curves $f_1=0$ and $f_2=0$ have the same ${\cal F}$ and the same multiplicity at the origin but in 
spite of it they have different topological type.

\section{Proof of Theorem \ref{decomposition}}
\label{proof-main-result}

\noindent Let $\tau$ be a positive rational number and let
$g(x,y)=\displaystyle \sum_{i\in \mathbf Q,j\in\bN}
a_{ij}x^iy^j\in\bC\{x\}^*[y]$. Put $w(x):=1$ and $w(y):=\tau$
the {\em weights} of the variables $x$ and $y$. By definition the
{\em weighted order} of $g$ is $\ord_{\tau}(g)=\min \{i+\tau j\;:\;
a_{ij}\neq 0\}$ and the {\em weighted initial part} of $g$ is
$\displaystyle \hbox{\rm in}_{\tau}(g)=\sum_{i+\tau
j=\ord_{\tau}(g)}a_{ij}x^iy^j$.

\medskip

\begin{Lemma}
\label{lemma1} Let $g(x,y)=u(x,y) \cdot x^N \prod_{i=1}^d
(y-\alpha_i(x))$, where $u(0,0)\neq 0$, $N\in\bQ$,
$\alpha_i(x)=c_ix^{\tau}+\cdots$ for $1\leq i\leq k$ and
$\ord_x(\alpha_i(x))<\tau$, for $k+1\leq i\leq d$. Then $\hbox{\rm
in}_{\tau}(g)=cx^M \prod_{i=1}^k(y-c_ix^{\tau})$ for some $c\in
\bC$ and some $M\in\bQ$.
\end{Lemma}

\noindent \begin{proof} Observe that $\hbox{\rm
in}_{\tau}(y-\alpha_i(x))=y-c_ix^{\tau}$ for $1\leq i\leq k$ and
$\hbox{\rm in}_{\tau}(y-\alpha_i(x))=-\hbox{\rm in}_{\tau}
\alpha_i(x)$ for $k+1\leq i\leq d.$ Since the initial part of a
product is the product of the initial parts of every factor we get
the lemma.
\end{proof}

\medskip

\begin{Lemma}
\label{lemma2} Let $h_1,h_2\in\bC\{x\}^*[y]$ and $\tau\in
\bQ^+$. Assume that the jacobian $\jac(\hbox{\rm
in}_{\tau}(h_1), \hbox{\rm in}_{\tau}(h_2))\neq 0$. Then
$\hbox{\rm in}_{\tau}(\jac(h_1,h_2))=\jac(\hbox{\rm in}_{\tau}(h_1),
\hbox{\rm in}_{\tau}(h_2)).$
\end{Lemma}

\noindent \begin{proof} For all monomials $M_1=x^{i_1}y^{j_1}$
and $M_2=x^{i_2}y^{j_2}$ we have $\jac(M_1,M_2)=(i_1j_2-i_2j_1)
x^{i_1+i_2-1}y^{j_1+j_2-1}$ hence $\hbox{\rm
ord}_{\tau}(\jac(M_1,M_2))=\ord_{\tau} (M_1)+\ord_{\tau}(M_2)-1-\tau$
provided 
$i_1j_2-i_2j_1\neq 0$. It follows that $\jac(\hbox{\rm
in}_{\tau}(h_1), \hbox{\rm in}_{\tau}(h_2))$ is the sum of monomials
of the same weighted order $\ord_{\tau}(\hbox{\rm
in}_{\tau}(h_1))+\ord_{\tau} (\hbox{\rm in}_{\tau}(h_2))-1-\tau$
(that is a quasi-homogeneous polynomial). Moreover $\jac(h_1,h_2)=
\jac(\hbox{\rm in}_{\tau}(h_1)+(h_1-\hbox{\rm in}_{\tau}(h_1)),
\hbox{\rm in}_{\tau}(h_2)+(h_2-\hbox{\rm
in}_{\tau}(h_2)))=\jac(\hbox{\rm in}_{\tau}(h_1),\hbox{\rm
in}_{\tau}(h_2))+$ {\em terms of higher weighted order} which proves the
lemma.
\end{proof}

\medskip

\noindent 
Recall that Newton-Puiseux roots of an irreducible Weierstrass 
polynomial~$f\in\bC\{x\}[y]$, $\deg f=n$ form a cycle:  
if $\gamma(x)=\sum a_ix^{i/n}$ is a root of $f$ then 
other roots of $f$ are $\gamma_j(x)=\sum a_i \omega_j^i x^{i/n}$, 
where $\omega_j$ is a $n$-th root of unity.  Moreover 
$\hbox{\rm ord}_x(\gamma(x)-\gamma_j(x))\geq \frac{b_{k+1}}{b_0}$ 
if and only if $\omega_j$ is a $l_k$-th root of unity (see \cite{Zariski}).

\medskip

\noindent Let $f=\prod_{i=1}^n(y-\gamma_i(x))$ be an irreducible 
Weierstrass polynomial with Puiseux characteristic $(b_0,\ldots,b_g)$ and
let $f^{(k)}(x,y)=\prod_{j=1}^{m}(y-\delta_j(x))$, where $n=ml_k$, 
be the characteristic approximate root of $f$. Put
$J(x,y):=\jac(f^{(k)},f)=\hbox{\rm unity}\cdot
x^{\alpha}\prod_l(y-\sigma_l(x))$. In order to prove Theorem
\ref{decomposition} we need 

\begin{Lemma}
\label{key-lemma} Fix $\gamma\in \Zer f$ and $\tau \in
\bQ$ such that $\tau\geq \frac{b_{k+1}}{b_0}$. Then
\begin{enumerate}
\item  if $\frac{b_j}{b_0}<\tau\leq\frac{b_{j+1}}{b_0}$,
          where $j\in\{k+1,\ldots,g\}$ 
          then $\sharp \{i\;:\;{\cal O}(\sigma_i,\gamma)\geq
\tau\}=l_j-1$,
\item if $\tau=\frac{b_{k+1}}{b_{0}}$ 
         then $\sharp \{i\;:\;{\cal O}(\sigma_i,\gamma)\geq \tau\}=n_{k+1}(l_{k+1}-1)$.
\end{enumerate}
\end{Lemma}

\noindent \begin{proof} 
Let
$\tilde{J}(x,y):=J(x,y+\gamma(x))$,
$\tilde{f}(x,y):=f(x,y+\gamma(x))$ and
$\tilde{f}^{(k)}(x,y):=f^{(k)}(x,y+\gamma(x))$. Clearly
$\tilde{J}(x,y)= \hbox{\rm unity}\cdot
x^{\alpha}\prod_l(y-(\sigma_l(x)-\gamma(x)))$. By Lemma~\ref{lemma1} 
$\sharp \{i\;:\;{\cal O}(\sigma_i,\gamma)\geq
\tau\}=\deg_y(\hbox{\rm in}_{\tau}(\tilde{J}(x,y)))$.

\medskip

\noindent Assume first that $\tau>\frac{b_{k+1}}{b_0}$ and $\tau\neq
\frac{b_j}{b_0}$ for all $j\in \{k+2,\ldots,g\}$. The weighted
initial part of
$\tilde{f}(x,y)=\prod_{i=1}^n(y-(\gamma_i(x)-\gamma(x)))$ is
equal to $\hbox{\rm in}_{\tau}(\tilde{f}(x,y))=c_1
x^{\alpha_1}y^{d(\tau)}$, where 
$c_1\in\bC\setminus\{0\}$ and 
$d(\tau):= \sharp \{i\;:\;{\cal O}(\gamma_i,\gamma)\geq \tau\}$. 
More precisely if $\frac{b_j}{b_0}<\tau
<\frac{b_{j+1}}{b_0}$ then $d(\tau)=l_j$.

\medskip

\noindent Consider the function
$\tilde{f}^{(k)}(x,y)=\prod_{j=1}^m(y-(\delta_j(x)-\gamma(x)))$.
Since ${\cal O}(\delta_j,\gamma)<\tau$ for every $j\in
\{1,\ldots,m\}$, we get by Lemma \ref{lemma1} 
$\hbox{\rm in}_{\tau}\tilde{f}^{(k)}(x,y)=c_2x^{\alpha_2}$,
where $c_2\in\bC \setminus\{0\}$.

\medskip

\noindent Using Lemma~\ref{lemma2} we get 
$$\hbox{\rm
in}_{\tau}(\tilde{J}(x,y))=\jac(c_2x^{\alpha_2},c_1x^{\alpha_1}y^{d(\tau)})=
c_1c_2\alpha_2 d(\tau) x^{\alpha_1+\alpha_2-1}y^{d(\tau)-1},$$ so
its $y$-degree is equal to $d(\tau)-1=l_j-1$ for
$\frac{b_j}{b_0}< \tau < \frac{b_{j+1}}{b_0}$.

\medskip

\noindent 
Let us choose $\tau < \frac{b_{j+1}}{b_0}$ 
close enough to $\frac{b_{j+1}}{b_0}$  that no $\sigma_i$ 
satisfies $\tau\leq{\cal O}(\sigma_i,\gamma)<\frac{b_{j+1}}{b_0}$.
 Then
$\sharp \{i\;:\;{\cal O}(\sigma_i,\gamma)\geq \tau\}=\sharp
\{i\;:\;{\cal O}(\sigma_i,\gamma)\geq \frac{b_{j+1}}{b_0}\}$
and the proof of statement~1  is done.

\medskip

\noindent Assume now that $\tau=\frac{b_{k+1}}{b_0}$. By Lemma
\ref{lemma1}
\begin{eqnarray*} \hbox{\rm
in}_{\tau}\tilde{f}(x,y)&=& x^{\alpha_3}\prod_{\omega^{l_k}=1}
(y-a(\omega^{b_{k+1}}-1)x^{b_{k+1}/b_0})\\
&=&x^{\alpha_3}\prod_{\omega^{l_k}=1}\left[
(y+ax^{b_{k+1}/b_0})-a\omega^{b_{k+1}}x^{b_{k+1}/b_0}\right]\\
&=&x^{\alpha_3}\left[
(y+ax^{b_{k+1}/b_0})^{n_{k+1}}-(ax^{b_{k+1}/b_0})^{n_{k+1}}
\right]^{l_{k+1}},
\end{eqnarray*}
where $\omega\in\bC$ and $a$
is the coefficient in $\gamma$ of the term $x^{b_{k+1}/b_0}$. The
last equality follows from the formula
$\prod_{\omega^p=1}(Z-b\omega^q)=\left(Z^{\frac{p}{\gcd(p,q)}}-
b^{\frac{p}{\gcd(p,q)}}\right)^{\gcd(p,q)}$.

\medskip

\noindent Moreover and also using Lemma \ref{lemma1} we have
$\hbox{\rm in}_{\tau}\tilde{f}^{(k)}(x,y)=x^{\alpha_4}(y+ax^{b_{k+1}/b_0})$
since there is only one Newton-Puiseux root $\delta_j$ of
$f^{(k)}$ such that ${\cal O}(\delta_j,\gamma)\geq
\frac{b_{k+1}}{b_0}$ (otherwise if there were two of such roots
$\delta_{j_1}$, $\delta_{j_2}$ then by the triangular property of the
contact order we obtain ${\cal O}(\delta_{j_1},\delta_{j_2})\geq
\frac{b_{k+1}}{b_0}$ which is not possible).

\bigskip

\noindent We prove now the equality $\alpha_3=\alpha_4
l_k$. Note that 
$\alpha_3=\sum_{i\in I'}{\cal O}(\gamma_i,\gamma)$ and 
$\alpha_4=\sum_{j\in J'}{\cal O}(\delta_j,\gamma)$, where 
$I':=\{i\;:\;{\cal O}(\gamma_i,\gamma)<\frac{b_{k+1}}{b_0}\}$  and
$J':=\{j\;:\;{\cal O}(\delta_j,\gamma)<\frac{b_{k+1}}{b_0}\}$.
Using Puiseux characteristic of $f$ and after Section 3 in
\cite{GP3} we obtain 
$\alpha_3=\sum_{i\in I'}{\cal O}(\gamma_i,\gamma)=
\sum_{l=1}^k \sharp\{i\;:\;{\cal O}(\gamma_i,\gamma)=\frac{b_{l}}{b_0}\}\cdot\frac{b_{l}}{b_0} = 
(n-l_1)\frac{b_1}{b_0}+\cdots+(l_{k-1}-l_k)\frac{b_k}{b_0}$
and by the same argument $\alpha_4=\sum_{j\in J'}{\cal
O}(\delta_j,\gamma)=\left(
\frac{n}{l_k}-\frac{l_1}{l_k}\right)\frac{b_1}{b_0}+\cdots+
\left(\frac{l_{k-1}}{l_k}-1\right)\frac{b_k}{b_0}.$

\medskip

\noindent Finally the initial part of $\tilde{J}$ is 
$$
\hbox{\rm in}_{\tau}(\tilde{J})=\jac(\hbox{\rm in}_{\tau}
(\tilde{f}^{(k)}),\hbox{\rm
in}_{\tau}(\tilde{f}))=\jac\left(v,(v^{n_{k+1}}-a^{n_{k+1}}u^{\theta
})^{l_{k+1}}\right),
$$
where $v=x^{\alpha_4}(y+ax^{b_{k+1}/b_0})$, $u=x$ and
$\theta=n_{k+1}\left(\frac{b_{k+1}}{b_0}+\alpha_4\right)$ so $\hbox{\rm
in}_{\tau}(\tilde{J})=\frac{\partial{\hbox{\rm
in}_{\tau}(\tilde{f})}}{\partial{u}}\frac{\partial v}{\partial y}$
and its $y$-degree is equal to $n_{k+1}(l_{k+1}-1)$.
\end{proof}

\begin{Remark} 
The proof of Merle formula in \cite{GP1} 
was based on the equality $\Delta_{\tilde{f}}=
\Delta_{\tilde{j}}+\{\Teis{\infty}{1}\}$, where
$\tilde{j}(x,y)=j(x,y+\gamma(x))$ and
$j(x,y):=\jac(x,f)$. Note that the statement of Lemma \ref{key-lemma}
can be written as $\deg_y \hbox{\rm in}_{\tau}(\tilde{J}(x,y))=\deg_y
\hbox{\rm in}_{\tau}(\tilde{f}(x,y))-1$ for
$\tau>\frac{b_{k+1}}{b_0}$. It follows from this equality that
 $\tilde{\Delta}_{\tilde{f}}=
\tilde{\Delta}_{\tilde{J}}+\{\Teis{\infty}{1}\}$, where 
$\tilde{\Delta}_{\tilde{J}}$ and  $\tilde{\Delta}_{\tilde{f}}$ are 
the sums of elementary Newton diagrams in the canonical decompositions
of $\Delta_{\tilde{J}}$ and  $\Delta_{\tilde{f}}$
respectively with inclinations bigger than $\frac{b_{k+1}}{b_0}$.

\end{Remark}

\begin{Corollary}
\label{corol} Keep the above notations and put
$\tau_i:=\cont(\sigma_i,f)$. Then
\begin{enumerate}
\item if $\tau_i\geq \frac{b_{k+1}}{b_0}$ then $\tau_i\in
\left\{\frac{b_{k+2}}{b_0},\ldots,\frac{b_{g}}{b_0}\right\}$.
\item The number
$\sharp\{i\;:\;\tau_i=\frac{b_{j}}{b_0}\}=n_1\cdots
n_{j-1}(n_j-1)\,$ for $j\in \{k+2,\ldots,g\}$.
\end{enumerate}
\end{Corollary}
\begin{proof}
First take $\tau$ such that $\frac{b_{j}}{b_0}<\tau\leq
\frac{b_{j+1}}{b_0}$ for $k+1\leq j\leq g$. We shall prove that

\begin{equation}
\label{cq} \sharp\{i\;:\;\tau_i\geq \tau\}=n-n_1\cdots n_j.
\end{equation}

\noindent In the set $\Zer f$ we define the equivalence relation given
by
$$\gamma^* \equiv \gamma'\; \hbox{\rm if and only if } \;{\cal
O}(\gamma^*,\gamma')\geq \frac{b_{j+1}}{b_0}.$$

\noindent Put $I_{\gamma}:=\{i\;:\;{\cal O}(\sigma_i,\gamma)\geq
\tau\}$ for $\gamma\in \Zer f$. By Lemma \ref{key-lemma} we get
$\sharp I_{\gamma}=l_{j}-1$. Note that
$I_{\gamma'}=I_{\gamma^*}$ for $\gamma^* \equiv \gamma'$ and
$I_{\gamma'}\cap I_{\gamma^*}=\emptyset$ when $\gamma^* \not\equiv
\gamma'$.

\medskip

\noindent Remark that $n_1\cdots n_j$ is the number of cosets in the
equivalence relation $\equiv$. Since $\sharp\{i\;:\;\tau_i\geq
\tau\}=\bigcup_{\gamma \in \Zer f}I_{\gamma}$ we have
$\sharp\{i\;:\;\tau_i\geq \tau\}=n_1\cdots n_j\cdot \sharp
I_{\gamma}=n_1\cdots n_j(l_{j}-1)=n-n_1\cdots n_j$. The
equality (\ref{cq}) is proved.

\medskip

\noindent Fix small positive number $\epsilon$ such that

\begin{eqnarray*}
\sharp\left\{i\;:\;\tau_i= \tau\right\}&=&\sharp\left\{i\;:\; \tau_i
\geq \tau\right\}- \sharp\left\{i\;:\; \tau_i\geq
\tau+\epsilon\right\}.\\
\end{eqnarray*}
If $\tau\neq \frac{b_j}{b_0}$ for all $j\in
\{k+2,\ldots,g\}$ the above difference is equal to zero. If
$\tau=\frac{b_j}{b_0}$ for some $j\in \{k+2,\ldots,g\}$, then 
$\sharp\left\{i\;:\;\tau_i= \frac{b_j}{b_0}\right\}=(n-n_1\cdots
n_{j-1})-(n-n_1\cdots n_{j})=n_1\cdots n_{j-1}(n_j-1)$.

\medskip

\noindent Finally using the same argument as before (for
$\tau=\frac{b_{k+1}}{b_0}$) we have

\begin{eqnarray*}
\sharp\left\{i\;:\;\tau_i=
\frac{b_{k+1}}{b_0}\right\}&=&\sharp\left\{i\;:\; \tau_i \geq
\frac{b_{k+1}}{b_0}\right\}- \sharp\left\{i\;:\; \tau_i\geq
\frac{b_{k+1}}{b_0}+\epsilon\right\}\\&=& \sharp \left\{i\;:\;
\tau_i \geq
\frac{b_{k+1}}{b_0}\right\}- (n-n_1\cdots n_{k+2})\\
&=&n_{k+1}(l_{k+1}-1)n_1\cdots n_{k}-(n-n_1\cdots n_{k+1})=0.
\end{eqnarray*}

\end{proof}

\bigskip

\noindent {\bf Proof of Theorem \ref{decomposition}.-} \noindent
Let $k+2\leq j\leq g$. Put $\Gamma^{(j)}=\prod (y-\sigma_i(x))$, where the product runs over $\sigma_i$ with $\cont(\sigma_i,f)=\frac{b_j}{b_0}$
 and let $\Gamma^{(k+1)}=\frac{\jac(f^{(k)},f)}{\Gamma^{(k+2)}\cdots \Gamma^{(g)}}$. 
It follows from the first statement  of Corollary \ref{corol} that for every Newton-Puiseux root
$\alpha \in \Zer \Gamma^{(k+1)}$ we have $\cont(\alpha,f)<\frac{b_{k+1}}{b_0}$. Finally by 
the second statement of Corollary \ref{corol} we get 
$(\Gamma^{(i)},x)_0=n_1\cdots n_{i-1}(n_i-1)$ for $k+2 \leq i\leq g$.

\section{Relation with Michel's theorem}

\medskip

\noindent In \cite{Michel} the author considered  a finite morphism
$(f,g):(X,p)\longrightarrow (\bC^2,0)$, where $(X,p)$ is a
normal germ of complex surface. Michel  determined the jacobian
quotients via a good minimal resolution and pointed out the
importance of the multiplicities of the jacobian quotients. More
precisely and following notation of \cite{Michel}, let $R$ be a good
resolution of $(f,g)$ and put $E=R^{-1}(p)$ the exceptional divisor
of $R$. For every irreducible component $E_i$ of $E$, denote $E'_i$
the set of points of $E_i$ which are smooth points of the total
transform ${\tilde E}=R^{-1}((fg)^{-1}(0))$. 
Denote the order of  $f\circ R$ (respectively $g\circ R$) at a generic point 
of $E_i$ $v(f,E_i)$ (respectively $v(f,E_i)$).
The quotient $q_i=\frac{v(g,E_i)}{v(f,E_i)}$ is the {\em
Hironaka number} of $E_i$.

\medskip
\noindent
Let $q$ be a Hironaka
number and put $E(q)$ the union of the $E'_i$ such that $q_i=q$ to
which we add $E_i\cap E_j$ if $q_i=q_j=q$. Let $\{E^k(q)\}_k$ be the
connected components of $E(q)$. By definition a $q$-{\em zone} is a
connected component of $E(q)$ and a $q$-zone is a {\em rupture zone}
if there exists in it at least one $E'_i$ with negative Euler characteristic.
Then after Theorem 4.8 of \cite{Michel} the set of
jacobian invariants of the morphism $(f,g)$ is equal to the set of
Hironaka numbers $q$ such that there exists at least one $q$-zone in
$E$ which is a rupture zone.

\bigskip

\noindent
Consider an irreducible Weierstrass polynomial $f$ 
with Puiseux characteristic $(b_0,b_1,\dots,b_g)$,
where $b_0<b_1$ (i.e. $x=0$ is transverse to $f=0$).
Below is the schematic picture of the resolution graph of the curve $f^{(k)}f=0$. 

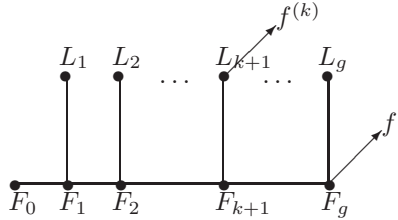
\begin{center}
\begin{picture}(100,70)(0,0)
 {\thicklines
\put(-20,0){\line(1,0){119}}}
\put(99,0){\vector(1,1){20}}\put(119,20){$f$}
\put(-22,-3){$\bullet$}\put(-22,-10){$F_0$}
\put(0,0){\line(0,1){40}} \put(-3,38){$\bullet$}\put(-3,45){$L_1$}
\put(-2,-3){$\bullet$}\put(-3,-10){$F_1$} \put(20,0){\line(0,1){40}}
\put(18,-3){$\bullet$}\put(17,-10){$F_2$}
\put(17,38){$\bullet$}\put(17,45){$L_2$} \put(59,0){\line(0,1){40}}
\put(57,-3){$\bullet$}\put(56,-10){$F_{k+1}$}    
\put(57,38){$\bullet$}\put(56,45){$L_{k+1}$}     
\put(59,40){\vector(1,1){20}}\put(79,60){$f^{(k)}$} 
\put(35,38){$\dots$} \put(99,0){\line(0,1){40}}          
\put(74,38){$\dots$} \put(99,0){\line(0,1){40}}          
\put(97,-3){$\bullet$}\put(96,-10){$F_g$}                   
\put(96,38){$\bullet$}\put(95,45){$L_g$}                     
\end{picture}
\end{center}

\bigskip

\noindent 
Every jacobian invariant 
$q\in\left\{\, l_k, \frac{l_{k+1}\overline{b_{k+2}}}{\overline{b_{k+1}}} ,\dots,
                      \frac{l_{g-1}\overline{b_g}}{\overline{b_{k+1}}}\,\right\}$
of the morphism $(f^{(k)},f)$ corresponds to exactly one rupture zone.

\noindent
The rupture zone for $q=l_k$ is the tree with endpoints $F_0$, $F_{k+1}$, $L_1$,\dots, $L_k$.
It yields the factor $\Gamma^{(k+1)}$ of the jacobian and by Michel's theorem 
$(\Gamma^{(k+1)},h)_0=\sum_{i=1}^{k+1} v(h,F_i)- \sum_{i=1}^{k}v(h,L_i)-v(h,F_0)$,
where $h=f$ or $h=f^{(k)}$.

\noindent
Every rupture zone for $q=\frac{l_{i-1}\overline{b_{i}}}{\overline{b_{k+1}}}$, where $k+2\leq i\leq g$
is the bamboo with endpoints $F_i$ and $L_i$. 
It yields the factor $\Gamma^{(i)}$ of the jacobian and by Michel's theorem 
$(\Gamma^{(i)},h)_0=v(h,F_i)-v(h,L_i)$ for $k+2\leq i\leq g$,
where $h=f$ or $h=f^{(k)}$.

\medskip
\noindent 
As an illustration we draw the resolution graph of $f^{(0)}f=0$,
where $f$ is the Weierstrass polynomial from Example \ref{example2}.
The labels of divisors are Hironaka numbers written 
in the form $\frac{v(f,E_i)}{v(f^{(0)},E_i)}$.

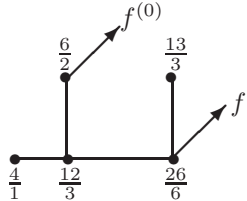
\bigskip
\begin{center}
\begin{picture}(100,50)(0,0)
 {\thicklines
\put(-20,0){\line(1,0){60}}
\put(40,0){\vector(1,1){20}}\put(62,18){$f$}
\put(-22,-3){$\bullet$}\put(-23,-12){$\frac{4}{1}$}
\put(0,0){\line(0,1){30}} \put(-3,28){$\bullet$}\put(-4,37){$\frac{6}{2}$}
\put(-2,-3){$\bullet$}\put(-4,-12){$\frac{12}{3}$} \put(40,0){\line(0,1){30}}
\put(38,-3){$\bullet$}\put(36,-12){$\frac{26}{6}$}
\put(37,28){$\bullet$}\put(36,37){$\frac{13}{3}$}
\put(0,30){\vector(1,1){20}}\put(20,50){$f^{(0)}$}}
\end{picture}
\end{center}

\bigskip
  
\noindent
There are two rupture zones corresponding to Hironaka numbers $4$ and $\frac{13}{3}$. 
It follows from \cite{Michel} that  
$\Nj(f^{(0)},f)=
\{\Teis{12}{3}\}-\{\Teis{4}{1}\}+\{\Teis{26}{6}\}-\{\Teis{13}{3}\}=
\{\Teis{8}{2}\}+\{\Teis{13}{3}\}$.

\bigskip
\begin{Remark}
\noindent Remark that Theorem \ref{decomposition} is also true when
we change $f^{(k)}$ for any irreducible Weierstrass polynomial with
the properties of statement of Proposition \ref{ch-prop}.
\end{Remark}

\medskip
\noindent
{\small Evelia Rosa Garc\'{\i}a Barroso\\
Departamento de Matem\'atica Fundamental\\
Facultad de Matem\'aticas, Universidad de La Laguna\\
38271 La Laguna, Tenerife, Espa\~na\\
e-mail: ergarcia@ull.es}

\medskip

\noindent {\small Janusz Gwo\'zdziewicz\\
Department of Mathematics\\
Technical University \\
Al. 1000 L PP7\\
25-314 Kielce, Poland\\
e-mail: matjg@tu.kielce.pl}

\begin{thebibliography}{9999999}

\bibitem[A-M] {Abhyankar-Moh} S.S. Abhyankar and T. Moh, {\em
Newton-Puiseux Expansions and Generalized Tschirnhausen
Transformation}, J. Reine Angew. Math. 260 (1973), 47-83; 261
(1973), 29-54.


\bibitem[Ca1] {Casas} E. Casas-Alvero, {\em Discriminant of a morphism and
inverse images of  plane curve singularities}, Math. Proc. Camb.
Phil. Soc. (2003), 135, 385-394.

\bibitem [Ca2] {Casas-Asian} Casas-Alvero, E. Local Geometry of planar analytic morphisms.
Asian J. Math. \textbf{11,} no. 3 (2007) 373-426.


\bibitem[Eph]{Ephraim} R. Ephraim, {\em Special polars and curves with one 
place at infinity}, Proceedings of Symposia in Pure Mathematics, 40, (1983), Part I, 353-359.

\bibitem[GB-G1]{GB-G} E.R. Garc\' \i a Barroso, J. Gwo\'zdziewicz,
{\em Characterization of jacobian Newton polygons of plane branches
and new criteria of irreducibility}, Annales Inst. Fourier, Tome 60,
n.2 (2010), 683-709.

\bibitem[GB-G2]{GB-G2} E.R. Garc\' \i a Barroso, J. Gwo\'zdziewicz,
{\em A discriminant criterion of irreducibility}, arXiv 0911.3771.


\bibitem[G-P\l 1]{GP1}    J. Gwo\'zdziewicz and A. P\l{}oski,
\textit {On the Merle formula for polar invariants,} Bull. Soc. Sci.
Lett. L\'odz 41 (7), (1991), 61-67.

\bibitem[G-P\l 2]{GP2}    J. Gwo\'zdziewicz and A. P\l{}oski,
                 \textit{On the Approximate Roots of polynomials,}
                 Annales Polonici Mathematici LX3 (1995), 199-210.

\bibitem[G-P\l 3]{GP3}    J. Gwo\'zdziewicz and A. P\l{}oski,
                 \textit{On the polar quotients of an analytic plane
                 curve,} Kodai Math. J. 25 (2002), 43-53.


\bibitem[Me] {Merle} M. Merle, {\em Invariants polaires des courbes
planes,} Invent. Math. 41, (1977) 103-111.

\bibitem[Mi] {Michel} F. Michel, {\em Jacobian curves for normal complex
surfaces}, Contemporary Mathematics, Volume 475 (2008), 135-150.

\bibitem[Po]{Popescu} P. Popescu-Pampu, {\em Approximate Roots},
Fields Institute Communications, Volume 33, (2003), 285-321.

\bibitem[Te1] {Teissier1} B. Teissier, {\em Variet\'es polaires.I. Invariants
polaires des singularit\'es des hypersurfaces}, Invent. Math. 40,
(1977) 267-292.

\bibitem[Te2]{Teissier2} B. Teissier, {\em The hunting of invariants
in the geometry of discriminants},  Proc. Nordic summer school,
1976. Per Holm, editor, Sijthoff and Noordhoff 1978, p.565-677.

\bibitem[Te3]{Teissier3} B. Teissier,  {\em Jacobian
Newton polyhedra and equisin\-gulari\-ty}, Proc. Kyoto Singularities
Symposium, RIMS, 1978.


\bibitem[Z] {Zariski} O. Zariski, {\em Le probl\`eme des modules pour
les branches planes,} Centre de Maths, Ecole Polytechnique, 1975.
Reprinted by Hermann, Paris 1986.

\end{thebibliography}
\end{document}